\documentclass[11pt]{amsart}
\usepackage{amssymb}
\newtheorem{theorem}{Theorem}
\newtheorem{remark}{Remark}

\begin{document}

\title [ Tur\'an Type Inequalities for The  $q$-exponential functions]{Tur\'an Type Inequalities for The  $q$-exponential functions\\}%

\author[ Khaled Mehrez]{  Khaled Mehrez }
\address{Khaled Mehrez. D\'epartement de Math\'ematiques ISSATK, Kairouan 3000, Tunisia.}
 \email{k.mehrez@yahoo.fr}

\begin{abstract}
In this paper our aim is to deduce some sharp  Tur\'an type inequalities for the remainder $q-$exponential functions. Our results are shown to be a generalization of results which were obtained by Alzer \cite{al}.
\end{abstract}
\maketitle
{\it keywords:} $q-$analogue of exponential functions, Tur\'an type inequalities.
\section{\textbf{Introduction}}
The inequalities of the type 
$$f_{n}(x)f_{n+2}(x)-f_{n+1}(x)^{2}\geq 0$$
have many applications in pure mathematics as in other branches of science. They
are named by Karlin and Szeg\"o \cite{sz}, Tur\'an-type inequalities because the first of these
type of inequalities was introduced in 1941 by P. Tur\'an \cite{t}.  More precisely, he used some results of Szeg\"o  \cite{sz1} to prove the previous inequality for $x\in(-1,1)$, where $f_n$ is the Legendre polynomial of degree $n$.  Actually, the Tur\'an type inequalities have a
more extensive literature and recently the results have been applied in problems arising from many fields such as information theory, economic theory and biophysics. 

Motivated by these applications, the Tur\'an type inequalities have recently come
under the spotlight once again and it has been shown that, for example, the classical
Gauss and Kummer hypergeometric functions, as well the generalized hypergeometric
functions, satisfy naturally some Tur\'an type inequalities \cite{ks, ks1, ks2}. For deep study about this subject we refer to \cite{al2,ba, ba1, kh}.

This paper is organized as follows: in Section 2 we present some preliminary results
and notations that will be useful in the sequel. In section 3,  we investigate some Tur\'an
type inequalities for the $q-$exponential functions.
\section{\textbf{Notations and preliminaries}}
Throughout this paper, we fix $q\in(0,1).$ We refer to \cite{gr} and \cite{ko} for the definitions, notations and properies of  the $q$-shifted factorials and the $q-$analogue of exponential functions. 
\subsection{Basic symbols}

Let $a\in\mathbb{R}$, the $q$-Schifted factorials are defined by
\[
(a;q)_{0}=1,\,\,\,\,\,(a;q)_{n}=\prod_{k=0}^{n-1}(1-aq^{k}),\,\,\,(a;q)_{\infty}=\prod_{k=0}^{\infty}(1-aq^{k}),\]
and we write
$$(a_{1},a_{2},...,a_{p};q)=(a_{1};q)_{n}(a_{1};q)_{n}...(a_{p};q)_{n},\;n=0,1,2,...$$
Note that for $q\longrightarrow1$ the expression $\frac{(q^{a};q)_{n}}{(1-q)^{n}}$ tend to $(a)_{n}=a(a+1)...(a+n-1).$
\subsection{$q$-analogue of exponential functions } 
For $q\in(0,1)$ the $q-$analogue of exponential function are given by \cite{gr,ko}
\begin{equation}
e(q;z)=\sum_{n=0}^{\infty}\frac{z^{n}}{(q,q)_{n}}=\frac{1}{(z;q)_{\infty}},\,\,|z|<1.
\end{equation}
and 
\begin{equation}
E(q;z)=\sum_{n=0}^{\infty}q^\frac{n(n-1)}{2}\frac{z^{n}}{(q,q)_{n}}=(-z;q)_{\infty}=\prod_{k=0}^{\infty}\left(1+zq^{k}\right),\;\;z\in\mathbb{C}.
\end{equation}
We denote by $I_n(q;z)$ and $\mathcal{I}_n(q;z)$ the differences 
\begin{equation}
I_n(q;z)=e(q;z)-\sum_{n=0}^{n}\frac{z^{n}}{(q,q)_{n}},\;0<z<1
\end{equation}
and 
\begin{equation}
\mathcal{I}_n(q;z)=e(q;z)-\sum_{n=0}^{n}q^\frac{n(n-1)}{2}\frac{z^{n}}{(q,q)_{n}},\;z>0.
\end{equation}
where $n$ is a nonnegative integer.

\section{\textbf{Tur\'an types inequalities for $q-$analogue of exponential functions}}
\begin{theorem}\label{T1} For every $n\in\mathbb{N},\;q\in(0,1)$  and $0<z<1$, The following Tur\'an type inequalities
\begin{equation}\label{1}
\frac{1-q^{n+1}}{1-q^{n+2}}\left(I_{n}(q;z)\right)^{2}<I_{n-1}(q;z)I_{n+1}(q;z)<\left(I_{n}(q;z)\right)^{2},
\end{equation}
\end{theorem}
hlods, where $\frac{1-q^{n+1}}{1-q^{n+2}}$ is the best possible constant. 
\begin{proof} Let $n\in\mathbb{N}$ and $q\in(0,1)$ we have 
\begin{equation}\label{01}
I_{n-1}(q;z)=I_{n}(q;z)+\frac{z^n}{(q,q)_n}
\end{equation}
and
\begin{equation}\label{02}
I_{n+1}(q;z)=I_{n}(q;z)-\frac{z^{n+1}}{(q,q)_n}.
\end{equation}
Thus 
\begin{equation}
\begin{split}
I_{n-1}(q;z)I_{n+1}(q;z)-I_{n}^{2}(q;z)&=I_{n}(q;z)\left(\frac{z^{n}}{(q,q)_{n}}-\frac{z^{n+1}}{(q,q)_{n+1}}\right)-\frac{z^{2n+1}}{(q,q)_{n}(q,q)_{n+1}}\\
&=\frac{z^{n}}{(q,q)_{n}}\left(\frac{z^{n+1}}{(q,q)_{n+1}}+\sum_{k=n+2}^{\infty}\frac{z^{k}}{(q,q)_{k}}\right)-\frac{z^{n+1}}{(q,q)_{n+1}}I_{n}(z,q)-\frac{z^{2n+1}}{(q,q)_{n}(q,q)_{n+1}}\\
&=\sum_{k=n+2}^{\infty}\frac{z^{k+n}}{(q,q)_{n}(q,q)_{k}}-\sum_{k=n+1}^{\infty}\frac{z^{k+n+1}}{(q,q)_{k}(q,q)_{n+1}}\\
&=\sum_{k=n+2}^{\infty}\frac{z^{k+n}}{(q,q)_{n}(q,q)_{k}}-\sum_{k=n+2}^{\infty}\frac{z^{k+n}}{(q,q)_{k-1}(q,q)_{n+1}}\\
&=\sum_{k=n+2}^{\infty}\left[\frac{(1-q^{n+1})-(1-q^{k})}{(q,q)_{n+1}(q,q)_{k}}\right]z^{k+n}\\
&=\sum_{k=n+2}^{\infty}\left[\frac{q^{k}-q^{n+1}}{(q,q)_{n+1}(q,q)_{k}}\right]z^{k+n}<0,
\end{split}
\end{equation}
from which follows the right hand  side inequality of (\ref{1}) for $0<z<1$ and $q\in(0,1).$\\
Now we prove the left hand side of (\ref{1}). From (\ref{01}) and (\ref{02}) we get the inequality 
\begin{equation}
\left[I_{n}(q;z)+\frac{z^{n}}{(q,q)_{n}}\right]\left[I_{n}(q;z)-\frac{z^{n+1}}{(q,q)_{n+1}}\right]>\frac{1-q^{n+1}}{1-q^{n+2}}\left(I_{n}(q;z)\right)^{2}
\end{equation}
which is equivalent to 
\begin{equation}
\frac{q^{n+1}-q^{n+2}}{1-q^{n+2}}\left(I_{n}(q;z)\right)^{2}>\frac{z^{2n+1}}{(q,q)_{n}(q,q)_{n+1}}+I_{n}(q;z)\left(\frac{z^{n+1}}{(q,q)_{n+1}}-\frac{z^{n}}{(q,q)_{n}}\right)
\end{equation}
On the other hand, we get 
\begin{equation}\label{001}
\begin{split}
\frac{q^{n+1}-q^{n+2}}{1-q^{n+2}}\left(I_{n}(q;z)\right)^{2}&=\left(\frac{q^{n+1}-q^{n+2}}{1-q^{n+2}}\right)\sum_{k=0}^{\infty}z^{2n+2+k}\left[\sum_{j=0}^{k}\frac{1}{(q,q)_{n+1+j}(q,q)_{n+1+k-j}}\right]\\
&>\left(\frac{q^{n+1}-q^{n+2}}{1-q^{n+2}}\right)\sum_{k=0}^{\infty}\frac{k+1}{(q,q)_{n+1}(q,q)_{n+k+1}}z^{2n+2+k}\\
&>(q^{n+1}-q^{n+2})\sum_{k=0}^{\infty}\frac{k+1}{(q,q)_{n+1}[(1-q^{n+2+k})(q,q)_{n+k+1}]}z^{2n+2+k}\\
&=(q^{n+1}-q^{n+2})\sum_{k=0}^{\infty}\frac{k+1}{(q,q)_{n+1}(q,q)_{n+k+2}}z^{2n+2+k}.
\end{split}
\end{equation}
Using the inequality 
\begin{equation}\label{w1}
\frac{q^{n+1}-q^{n+k+2}}{q^{n+1}-q^{n+2}}=\frac{1-q^{k+1}}{1-q}\leq k+1, \textrm{for\;all}\;k\in\mathbb{N} \textrm{and} q\in(0,1)
\end{equation}
and (\ref{001}) we obtain
\begin{equation} \label{k0}
\frac{q^{n+1}-q^{n+2}}{1-q^{n+2}}\left(I_{n}(q;z)\right)^{2}>\sum_{k=0}^{\infty}\frac{\left[q^{n+1}-q^{n+k+2}\right]}{(q,q)_{n+1}(q,q)_{n+k+2}}z^{2n+2+k}.
\end{equation}
Also, Using the inequalities 
\begin{equation}
\frac{z^{n+1}}{(q,q)_{n+1}}I_{n}(q;z)=\sum_{k=0}^{\infty}\frac{z^{2n+2+k}}{(q,q)_{n+1}(q,q)_{k+n+1}}
\end{equation}
and 
\begin{equation}
\begin{split}
\frac{z^{n}}{(q,q)_{n}}I_{n}(q;z)&=\sum_{k=0}^{\infty}\frac{z^{2n+1+k}}{(q,q)_{n}(q,q)_{k+n+1}}\\
&=\frac{z^{2n+1}}{(q,q)_{n}(q,q)_{n+1}}+\sum_{k=0}^{\infty}\frac{z^{2n+2+k}}{(q,q)_{n}(q,q)_{k+n+2}}
\end{split}
\end{equation}
we obtain
\begin{equation}\label{k1}
\begin{split}
I_{n}(q;z)\left(\frac{z^{n+1}}{(q,q)_{n+1}}-\frac{z^{n}}{(q,q)_{n}}\right)+\frac{z^{2n+1}}{(q,q)_{n}(q,q)_{n+1}}&=\sum_{k=0}^{\infty}\left[\frac{1}{(q,q)_{n+1}(q,q)_{n+k+1}}-\frac{1}{(q,q)_{n}(q,q)_{n+k+2}}\right]z^{2n+2+k}\\
&=\sum_{k=0}^{\infty}\left[\frac{(1-q^{n+k+2})-(1-q^{n+1})}{(q,q)_{n+1}(q,q)_{n+k+2}}\right]z^{2n+2+k}\\
&=\sum_{k=0}^{\infty}\left[\frac{q^{n+1}-q^{n+k+2}}{(q,q)_{n+1}(q,q)_{n+k+2}}\right]z^{2n+2+k}.
\end{split}
\end{equation}
Combining (\ref{k0}) and (\ref{k1}) we get the left hand side inequality of \ref{1}.
Finally, since 
\[\lim_{z\longrightarrow0}\frac{I_{n-1}(q;z)I_{n+1}(q;z)}{\left(I_{n}(z;q)\right)^{2}}=\frac{1-q^{n+1}}{1-q^{n+2}}\]
we conclude that in inequality (\ref{1}) the value $\frac{1-q^{n+1}}{1-q^{n+2}}$ is the best possible constant.\\
So the proof of Theorem \ref{T1} is complete.
\end{proof}
\begin{theorem}\label{T2} For all $q\in(0,1)$ and for every $n\in\mathbb{N}$ the following Tur\'an type inequalities
\begin{equation}\label{10}
\left[\frac{q-q^{n+2}}{1-q^{n+2}}\right]\left(\mathcal{I}_{n}(q;z)\right)^{2}\leq\mathcal{I}_{n-1}(q;z)\mathcal{I}_{n+1}(q;z)<\left(\mathcal{I}_{n}(q;z)\right)^{2},
\end{equation}
hold for all $z>0.$ The value $\left[\frac{q-q^{n+2}}{1-q^{n+2}}\right]$ is the best possible constant.
\end{theorem}
\begin{proof} Let $q\in(0,1),\;n\in\mathbb{N}$ and $z>0$, using the inequalities
\begin{equation}\label{kh1}
\mathcal{I}_{n-1}(q;z)=\mathcal{I}_{n}(q;z)+\frac{q^{\frac{n(n-1)}{2}}}{(q;q)_{n}}z^{n}
\end{equation}
and
\begin{equation}\label{kh2}
\mathcal{I}_{n+1}(q;z)=\mathcal{I}_{n}(q;z)-\frac{q^{\frac{n(n+1)}{2}}}{(q;q)_{n}}z^{n}
\end{equation}
we obtain
\begin{equation}\label{11}
\begin{split}
\mathcal{I}_{n+1}(q;z)\mathcal{I}_{n-1}(q;z)-\left(\mathcal{I}_{n}(q;z)\right)^{2}&=\mathcal{I}_{n}(q;z)\left[\frac{q^{\frac{n(n-1)}{2}}}{(q;q)_{n}}z^{n}-\frac{q^{\frac{n(n+1)}{2}}}{(q;q)_{n+1}}z^{n+1}\right]-\frac{q^{n^2}z^{2n+1}}{(q;q)_{n}(q;q)_{n+1}}\\
&=\frac{q^{\frac{n(n-1)}{2}}}{(q;q)_{n}}\sum_{j=n+2}^{\infty}\frac{q^{\frac{j(j-1)}{2}}}{(q;q)_{j}}z^{j+n}-\frac{q^{\frac{n(n+1)}{2}}}{(q;q)_{n+1}}\sum_{j=n+1}^{\infty}\frac{q^{\frac{j(j-1)}{2}}}{(q;q)_{j}}z^{j+n+1}\\
&=\frac{q^{\frac{n(n-1)}{2}}}{(q;q)_{n}}\sum_{j=n+2}^{\infty}\frac{q^{\frac{j(j-1)}{2}}}{(q;q)_{j}}z^{j+n}-\frac{q^{\frac{n(n+1)}{2}}}{(q;q)_{n+1}}\sum_{j=n+2}^{\infty}\frac{q^{\frac{(j-2)(j-1)}{2}}}{(q;q)_{j-1}}z^{j+n}\\
&=\sum_{j=n+2}^{\infty}\frac{q^{\frac{n(n-1)+(j-1)(j-2)}{2}}\left[q^{j-1}-q^{n}\right]}{(q;q)_{n+1}(q;q)_{j}}x^{j+n}<0.
\end{split}
\end{equation}
From (\ref{11}) we obtain the right side inequality of (\ref{10}) for all $q\in(0,1)$ and $x>0.$\\

The inequality 
\begin{equation}
\left[\frac{q-q^{n+2}}{1-q^{n+2}}\right]\left(\mathcal{I}_{n}(q;z)\right)^{2}\leq\mathcal{I}_{n-1}(q;z)\mathcal{I}_{n+1}(q;z)
\end{equation}
is equivalent to
\begin{equation}\label{tt02}
\left(\frac{1-q}{1-q^{n+2}}\right)\left(\mathcal{I}_{n}(q;z)\right)^{2}\geq\frac{q^{n^{2}}z^{2n+1}}{(q;q)_{n}(q;q)_{n+1}}+\mathcal{I}_{n}(q;z)\left(\frac{q^{\frac{n(n+1)}{2}}}{(q;q)_{n+1}}z^{n+1}-\frac{q^{\frac{n(n-1)}{2}}}{(q;q)_{n}}z^{n}\right).
\end{equation}
So
\begin{equation}\label{t01}
\frac{q^{n^{2}}z^{2n+1}}{(q;q)_{n}(q;q)_{n+1}}+\mathcal{I}_{n}(q;z)\left(\frac{q^{\frac{n(n+1)}{2}}}{(q;q)_{n+1}}z^{n+1}-\frac{q^{\frac{n(n-1)}{2}}}{(q;q)_{n}}z^{n}\right)=
\end{equation}
\begin{equation*}
\begin{split}
\;\;\;\;\;\;\;\;\;\;&=\sum_{k=n+1}^{\infty}\frac{q^{\frac{k(k-1)+n(n+1)}{2}}z^{k+n+1}}{(q;q)_{n+1}(q;q)_{k}}-\sum_{k=n+2}^{\infty}\frac{q^{\frac{k(k-1)+n(n-1)}{2}}z^{k+n}}{(q;q)_{n}(q;q)_{k}}\\
&=q^{\frac{n(n+1)}{2}}\sum_{k=0}^{\infty}\frac{q^{\frac{(n+k)(n+1+k)}{2}}}{(q;q)_{n+1}(q;q)_{n+1+k}}z^{2n+2+k}-q^{\frac{n(n-1)}{2}}\sum_{k=0}^{\infty}\frac{q^{\frac{(n+k+1)(n+2+k)}{2}}}{(q;q)_{n}(q;q)_{n+2+k}}z^{2n+2+k}\\
&=\sum_{k=0}^{\infty}\frac{q^{\frac{n(n-1)+(n+k)(n+k+1)}{2}}}{(q;q)_{n}(q;q)_{n+k+1}}\left[\frac{q^{n}}{1-q^{n+1}}-\frac{q^{n+k+1}}{1-q^{n+k+2}}\right]z^{2n+2+k}\\
&=\sum_{k=0}^{\infty}\frac{q^{\frac{n(n+1)+(n+k)(n+k+1)}{2}}}{(q;q)_{n+1}(q;q)_{n+k+2}}\left(1-q^{k+1}\right)z^{2n+2+k}\\
&=\sum_{k=0}^{\infty}\frac{q^{\frac{n(n+1)+(n+k)(n+k+1)}{2}}}{(q;q)_{n+1}(q;q)_{n+k+2}}\left(1-q^{k+1}\right)z^{2n+2+k}.
\end{split}
\end{equation*}
On the other hand, since 
\begin{equation*}
(n+j+1)(n+j)+(n+k-j+1)(n+k-j)\leq n(n+1)+(n+k)(n+k+1)
\end{equation*}
for all $n\in\mathbb{N}$ and $0\leq j\leq k$, we get 
\begin{equation}
\begin{split}
\left(\frac{1-q}{1-q^{n+2}}\right)\left(\mathcal{I}_{n}(q;z)\right)^{2}&=\left(\frac{1-q}{1-q^{n+2}}\right)\sum_{k=0}^{\infty}z^{2n+2+k}\sum_{j=0}^{k}\frac{q^{\frac{(n+j)(n+j+1)+(n+k-j)(n+k-j+1)}{2}}}{(q;q)_{n+1+j}(q;q)_{n+k-j+1}}\\
&\geq\left(\frac{1-q}{1-q^{n+2}}\right)\sum_{k=0}^{\infty}\frac{z^{2n+2+k}}{(q;q)_{n+1}(q;q)_{n+k+1}}\sum_{j=0}^{k}q^{\frac{(n+j)(n+j+1)+(n+k-j)(n+k-j+1)}{2}}\\
&\geq\left(\frac{1-q}{1-q^{n+2}}\right)\sum_{k=0}^{\infty}\frac{z^{2n+2+k}(k+1)q^{\frac{n(n+1)+(n+k)(n+k+1)}{2}}}{(q;q)_{n+1}(q;q)_{n+k+1}}\\
&\geq (1-q)\sum_{k=0}^{\infty}\frac{z^{2n+2+k}(k+1)q^{\frac{n(n+1)+(n+k)(n+k+1)}{2}}}{(q;q)_{n+1}(q;q)_{n+k+2}}
\end{split}
\end{equation}
Now, from the previous inequality , (\ref{w1}) and (\ref{t01}), we have 
\begin{equation}
\begin{split}
\left(\frac{1-q}{1-q^{n+2}}\right)\left(\mathcal{I}_{n}(q;z)\right)^{2}&\geq \sum_{k=0}^{\infty}\frac{q^{\frac{n(n+1)+(n+k)(n+k+1)}{2}}}{(q;q)_{n+1}(q;q)_{n+k+2}}(1-q^{k+1})z^{2n+2+k}\\
&=\frac{q^{n^{2}}z^{2n+1}}{(q;q)_{n}(q;q)_{n+1}}+\mathcal{I}_{n}(q;z)\left(\frac{q^{\frac{n(n+1)}{2}}}{(q;q)_{n+1}}z^{n+1}-\frac{q^{\frac{n(n-1)}{2}}}{(q;q)_{n}}z^{n}\right).
\end{split}
\end{equation}
So, the inequalitty (\ref{tt02}) holds, from which follows the left side inequality of (\ref{10}) for all $q\in(0,1)$ and $x>0.$\\
Furthermore,
\[\lim_{z\longrightarrow0}\frac{\mathcal{I}_{n-1}(q;z)\mathcal{I}_{n+1}(q;z)}{\left(\mathcal{I}_{n}(z;q)\right)^{2}}=\frac{q-q^{n+2}}{1-q^{n+2}}.\]
The proof of Theorem \ref{T2} is completed.
\end{proof}
\begin{remark} Observe also that if $q$ tends to $1$ in Theorem \ref{T2}, then we get the following result: If $n\in\mathbb{N}$ and $x>0$, then the Tur\'an type inequality:
\begin{equation*}
\frac{n+1}{n+2}I_n^2(x))<I_{n-1}(x)I_{n+1}(x)
\end{equation*}
is valid. We note that this inequality was proved by H. Alzer \cite{al}
\end{remark}

\end{document}